\documentclass[11pt,reqno]{amsart}
\usepackage[all]{xy}
\usepackage{amssymb}
\usepackage{amsthm}
\usepackage[normalem]{ulem}
\usepackage{amsmath,mathtools}
\usepackage{amscd,enumitem}
\usepackage{caption}
\usepackage[justification=centering]{caption}
\usepackage{verbatim}
\usepackage{eurosym}
\usepackage{float}
\usepackage{color}
\usepackage{dcolumn}
\usepackage[mathscr]{eucal}
\usepackage[all]{xy}
\usepackage{bbm}
\usepackage[textheight=8.7in, textwidth=6.7in]{geometry}
\newtheorem*{conj*}{Conjecture}
\newtheorem{theorem}{Theorem}[section]

\theoremstyle{definition}
\newtheorem*{remark}{Remark}
\newtheorem*{tworemarks}{Two Remarks}
\theoremstyle{plain}

\newtheorem{prop}[theorem]{Proposition}

\newtheorem{corollary}[theorem]{Corollary}
\newcommand{\preg}{p_{\mathrm{reg}}}

\newcommand{\Q}{\mathbb{Q}}
\newcommand{\R}{\mathbb{R}}
\newcommand{\N}{\mathbb{N}}

\newcommand{\C}{\mathbb{C}}

\renewcommand{\pmod}[1]{\,\,({\rm mod}\,\,{#1})}

\numberwithin{equation}{section}

\newtheoremstyle{example}
  {\topsep}   
  {\topsep}   
  {\normalfont}  
  {0pt}       
  {\bfseries} 
  {.}         
  {5pt plus 1pt minus 1pt} 
  {}          
\theoremstyle{example}
\newtheorem*{example}{Example}

\def\({\left(}
\def\){\right)}

\usepackage{centernot}

\begin{document}
\title[Limiting Betti distributions]{Limiting Betti distributions of Hilbert schemes on $n$ points}
\author{Michael Griffin, Ken Ono, Larry Rolen, and Wei-Lun Tsai}

\address{Department of Mathematics, Brigham Young University, Provo, UT 84602}
\email{mjgriffin@math.byu.edu}

\address{Department of Mathematics, University of Virginia, Charlottesville, VA 22904}
\email{ko5wk@virginia.edu}

\address{Department of Mathematics, Vanderbilt University, Nashville, TN 37240}
\email{larry.rolen@vanderbilt.edu}

\address{Department of Mathematics, University of Virginia, Charlottesville, VA 22904}
\email{wt8zj@virginia.edu}

\keywords{Betti numbers, Hilbert schemes, partitions}
\subjclass[2020]{14C05, 14F99, 11P82}

\begin{abstract} 
 Hausel and Rodriguez-Villegas \cite{HRV} recently observed that work of G\"ottsche, combined with a classical result of Erd\H{o}s and Lehner
on integer partitions, implies that the
  limiting Betti distribution for 
the Hilbert schemes $(\C^2)^{[n]}$ on $n$ points, as $n\rightarrow +\infty,$ is a {\it Gumbel distribution}. In view of this example, they ask for further
such Betti distributions. We answer this question for
  the quasihomogeneous Hilbert schemes $((\C^2)^{[n]})^{T_{\alpha,\beta}}$ that are cut out by torus actions.
 We prove that their limiting distributions are also of Gumbel type. 
 To obtain this result, we combine work of Buryak, Feigin, and Nakajima on these Hilbert schemes with our
 generalization of the result of Erd\H{o}s and Lehner, which gives the distribution of
 the number of parts in partitions that are multiples of a fixed integer $A\geq 2.$
 Furthermore, if $p_k(A;n)$ denotes the number of partitions of $n$ with exactly $k$ parts that are multiples of $A$, then we obtain the asymptotic
  $$
 p_k(A,n)\sim \frac{24^{\frac k2-\frac14}(n-Ak)^{\frac k2-\frac34}}{\sqrt2\left(1-\frac1A\right)^{\frac k2-\frac14}k!A^{k+\frac12}(2\pi)^k}e^{2\pi\sqrt{\frac1{6}\left(1-\frac1A\right)(n-Ak)}},
 $$
 a result which is of independent interest.
\end{abstract}

\maketitle
\section{Introduction and statement of results}

We consider the Hilbert schemes of $n$ points on $\C^2,$ denoted
 $X^{[n]}=(\C^2)^{[n]},$ that have been
  studied by G\"ottsche \cite{Gottsche, GottscheICM}, and Buryak, Feigin, and Nakajima \cite{BuryakFeigin, IMRN}. 
Each $X^{[n]}$ is a nonsingular, irreducible, quasiprojective dimension $2n$ algebraic variety.
Moreover, they enjoy the convenient description
\begin{equation}
X^{[n]}=\left \{ I \subset \C[x,y] \ : \ {\text {\rm $I$ is an ideal with $\dim_{\C}(\C[x,y]/I)=n$}}\right\},
\end{equation}
which reduces the calculation of its Betti numbers to problems on integer partitions. 
To investigate these Betti numbers, it is natural to combine them to form
 the generating function
\begin{equation}
P\left(X^{[n]};T\right):=\sum_{j=0}^{2n-2} b_j(n) T^{j}=
\sum_{j=0}^{2n-2} \dim \left(H_j\left(X^{[n]},\Q\right)\right)T^j,
\end{equation}
known as its {\it Poincar\'e polynomial}.
Due to the connection with integer partitions,  it turns out that these polynomial generating functions equivalently keep track of the number of parts among the size $n$ partitions.

In their work on the statistical properties of certain varieties, Hausel and Rodriguez-Villegas \cite{HRV} observed that
a classical result of Erd\H{o}s and Lehner on partitions \cite{EL}
gives (see Section 4.3 of \cite{HRV})
 the limiting distribution for the Betti numbers of $X^{[n]}$ as $n\rightarrow +\infty$. Using
G\"ottsche's generating function \cite{Gottsche, GottscheICM} for the $P(X^{[n]};T),$  it is straightforward to compute examples that offer glimpses of this result.
For example, we find that
$$
P\left(X^{[50]};T\right)=1+T^2+2T^4+\dots+5427T^{88}+2611T^{90}+920T^{92}+208T^{94}+25T^{96}+T^{98}.
$$
The renormalized even degree\footnote{The coefficients $b_{2j+1}(n)$ for odd degree terms identically vanish.}  coefficients are plotted in Figure~\ref{figure1}. As $P\left(X^{[50]};1\right)=p(50),$ the number of partitions of $50,$ the plot consists of the points
$\left \{ \left(\frac{2m}{98}, \frac{b_{2m}(50)}{p(50)}\right) \ : \ 0\leq m\leq 49\right\}.$

 \smallskip
 \begin{center}
\includegraphics[height=65mm]{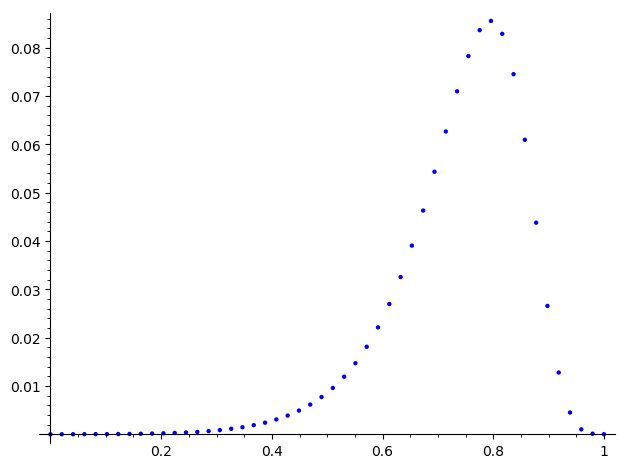}
\captionof{figure}{Betti distribution for $X^{[50]}$}\label{figure1}
\end{center}
\smallskip

\noindent
These distributions, when properly renormalized, converge to a {\it Gumbel distribution} as $n\rightarrow +\infty.$ 

Hausel and Rodriguez-Villegas asked for further such $n$-aspect Betti distributions.
We answer this question for the
quasihomogeneous $n$ point Hilbert schemes that are cut out by torus actions.  To define them, we use the torus  $(\C^{\times})^2$-action on $\C^2$ defined by scalar multiplication 
$$(t_1, t_2)\cdot (x,y):=(t_1x, t_2 y),
$$
which  lifts to $X^{[n]}=(\C^2)^{[n]}.$ For relatively prime $\alpha, \beta\in \N,$ we 
have the one-dimensional subtorus
$$T_{\alpha,\beta}:=\{(t^{\alpha}, t^{\beta}) \ : \ t\in \C^{\times}\}.
$$ 
The quasihomogeneous Hilbert scheme $X^{[n]}_{\alpha,\beta}:=((\C^2)^{[n]})^{T_{\alpha,\beta}}$ is the fixed point set
of  $X^{[n]}.$  

To define Betti distributions, we make use of the Poincar\'e polynomials
\begin{equation}\label{PoincareDefn}
P\left(X^{[n]}_{\alpha,\beta};T\right):=\sum_{j=0}^{2\lfloor \frac{n}{\alpha+\beta}\rfloor} b_j(\alpha,\beta; n) T^j=
\sum_{j=0}^{2\lfloor \frac{n}{\alpha+\beta}\rfloor} \dim \left(H_j\left(X^{[n]}_{\alpha,\beta},\Q\right)\right)T^j.
\end{equation}
As 
$P\left(X^{[n]}_{\alpha,\beta};1\right)=p(n),$ we have that the discrete measure
$d\mu^{[n]}_{\alpha,\beta}$
for $X^{[n]}_{\alpha,\beta}$ is
\begin{equation}\label{discretemeasureequation}
\Phi_n(\alpha, \beta; x):=\frac{1}{p(n)}\cdot \int_{-\infty}^x d\mu^{[n]}_{\alpha,\beta}=
\frac{1}{p(n)}\cdot \sum_{j\leq x} b_j(\alpha,\beta;n).
\end{equation}
The following theorem gives the limiting Betti distributions (as functions in $x$) we seek.

\begin{theorem}\label{MainTheorem}
If $\alpha$ and $\beta$ are relatively prime positive integers, then
$$
\lim_{n\rightarrow +\infty} \Phi_n(\alpha, \beta; 2\sqrt{n} x+\delta_n(\alpha,\beta))=\exp\left (
-\frac{\sqrt{6}}{\pi(\alpha+\beta)}\cdot  \exp\left( -\frac{\pi(\alpha+\beta)}{\sqrt{6}}x 
\right)\right),
$$
where  $\delta_n(\alpha,\beta):=\frac{\sqrt{6}}{\pi(\alpha+\beta)}\sqrt{n}\log(n).$
\end{theorem}

\begin{tworemarks}
\

\noindent
(1) The limiting cumulative distribution in Theorem~\ref{MainTheorem} is of Gumbel type
\cite{Gumbel1, Gumbel2}. Such distributions are often used to study the maximum (resp. minimum) of a 
number of samples of a random variable. Letting $A:=\alpha+\beta,$ we have mean $\frac{\sqrt{6}}{A\pi}\left( \log\left(\frac{\sqrt{6}}{A\pi}\right)+\gamma \right),$ where $\gamma$ is the Euler-Mascheroni constant, and variance $1/A^2.$

\smallskip
\noindent
(2) Gillman, Gonzalez, Schoenbauer and two of the authors studied a different kind of distribution for Hilbert schemes of surfaces in \cite{GGOR}. In that work equidistribution results were obtained for the Hodge numbers organized by congruence conditions.
\end{tworemarks}

\begin{example} For example, let $\alpha=1$ and $\beta=2.$ 
For $n=20,$ we have
$$
P\left(X^{[20]}_{1,2};T\right)=202+212T^2+126T^4+56T^6+22T^8+7T^{10}+2T^{12}.
$$
This small degree polynomial  is not very suggestive. However, for $n=1000$
the renormalized even degree\footnote{The odd degree coefficients terms identically vanish.} coefficients displayed in Figure~\ref{figure2} is quite illuminating. As $P\left(X^{[1000]}_{1,2};1\right)=p(1000),$  the plot consists of the  334 points
$\left \{ \left(\frac{2m}{666}, \frac{b_{2m}(1000)}{p(1000)}\right) \ : \ 0\leq m\leq 333\right\}.$

 \smallskip
 \begin{center}
\includegraphics[height=70mm]{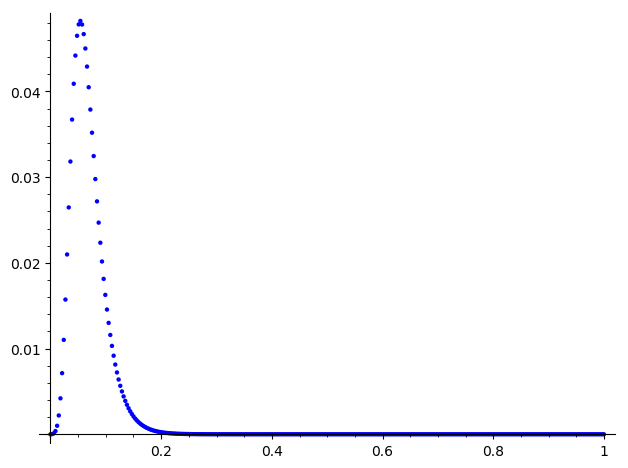}
\captionof{figure}{Betti distribution for $X^{[1000]}_{1,2}$}\label{figure2}
\end{center}
Theorem~\ref{MainTheorem} gives the cumulative distribution corresponding to such plots as $n\rightarrow +\infty.$ In this case, the theorem asserts that
$$
\lim_{n\rightarrow +\infty} \Phi_n\left (1, 2; 2\sqrt{n}x+\frac{\sqrt{6n}}{3\pi}\cdot \log(n)\right)=\exp\left (
-\frac{\sqrt{6}}{3\pi}\cdot  \exp\left( -\frac{3\pi x}{\sqrt{6}}
\right)\right).
$$
\smallskip
\end{example}

Theorem~\ref{MainTheorem} follows from a result which is of independent interest that generalizes a theorem of Erd\H{o}s and
Lehner on the distribution of the number of parts in partitions of fixed size.
Using the
celebrated Hardy-Ramanujan asymptotic formula
 $$
 p(n)\sim \frac{1}{4n\sqrt{3}}\cdot \exp(C\sqrt{n}),
 $$
 where  $C:=\pi \sqrt{2/3},$
 Erd\H{o}s and Lehner determined the distribution of the number of parts in partitions of size $n$.
 More precisely, if $k_n=k_n(x):=C^{-1} \sqrt{n} \log(n)+\sqrt{n}x,$ they proved (see Theorem 1.1 of \cite{EL}) that
 \begin{equation}\label{ELTheorem}
 \lim_{n\rightarrow +\infty} \frac{p_{\leq k_n}(n)}{p(n)}=\exp\left(-\frac{2}{C} e^{-\frac{1}{2}C x}\right),
 \end{equation}
 where $p_{\leq k}(n)$ denotes\footnote{We note that $p_{\leq k}(n)$ is denoted $p_k(n)$ in \cite{EL}.} the number of partitions of $n$ with at most $k$ parts.
 In particular, the normal order for the number of parts of a partition of size $n$ is $C^{-1} \sqrt{n} \log(n).$
 
 To prove Theorem~\ref{MainTheorem}, the generalization of the observation of Hausel and Rodriguez-Villegas, we require the distribution of the number of parts in partitions
 that are multiples of a fixed integer $A\geq 2$.   The next theorem describes these distributions.
 
 \begin{theorem}\label{ELGeneralization} If $A\geq 2$ and $p_{\leq k}(A;n)$ denotes the number of partitions of $n$ with at most $k$ parts that are multiples of $A$, then for
$ k_{A,n}=k_{A,n}(x):=\frac{1}{AC} \sqrt{n} \log(n) +x \sqrt{n},$
 we have
 $$
 \lim_{n\rightarrow +\infty} \frac{p_{\leq k_{A,n}}(A;n)}{p(n)}=\exp\left (
-\frac{2}{AC} \exp\left( -\frac 12 xAC 
\right)\right).
 $$
 \end{theorem}

\begin{remark} The distribution functions in Theorem~\ref{ELGeneralization} are of Gumbel type with mean $\frac{2}{AC}\left( \log\left(\frac{2}{AC}\right)+\gamma \right)$ and variance $1/A^2.$
\end{remark}

\begin{example} Here we illustrate Theorem~\ref{ELGeneralization} with $A=2$ and $n=600.$
In this case we have
\begin{displaymath}
k_{2,600}(x):= \frac{\sqrt{600} \log(600)}{2C} +\sqrt{600}x.
\end{displaymath}
For real numbers $x,$ we let
$$
\delta_{k_{2,600}}(x):=\frac{\# \{ {\text {\rm $\lambda \vdash 600$ with $\leq k_{2,n}(x)$ many even parts}}\} }{p(n)}.
$$
The theorem indicates that these proportions are approximated by the Gumbel values
$$
G_{2,600}(x):=\exp \left(-\frac{1}{C}\cdot e^{-Cx}\right).
$$
The table below illustrates the strength of these approximations for various values of $x$.
\smallskip

\begin{center} 
\begin{small}
\begin{tabular}{|c|c|c|c|}
 \hline 
$x$ &$\lfloor k_{2,600}(x)\rfloor$& $\delta_{k_{2,600}}(x)$ & $G_{2, 600}(x)$ \\ \hline \hline
$-0.1$ & $28$ &  $0.597\dots$ & $0.604\dots$ \\ \hline
$\ 0.0$ &$30$ & $0.663\dots$ & $0.677\dots$ \\ \hline
$\ 0.1$ & $32$ & $0.721\dots$ & $0.739\dots$ \\ \hline
$\ 0.2$ & $35$ & $0.791\dots$ & $0.792\dots$ \\ \hline
$\ 0.3$ & $37$ & $0.830\dots$ & $0.835\dots$ \\ \hline
$\vdots$ & $\vdots$ & $\vdots$ & $\vdots$ \\ \hline
$\ 1.5$ & $67$ & $0.994\dots$ & $0.992\dots$ \\ \hline
$\ 2.0$ & $79$ & $0.998\dots$ & $0.998\dots$ \\ \hline
\end{tabular}
\end{small}
\end{center}
\end{example}
\smallskip

We note that
Theorem~\ref{ELGeneralization} does not offer the asymptotics for
 $p_k(A;n),$ the number of partitions of $n$ with exactly $k$ parts that are multiples of $A$.
 For completeness, we offer such asymptotics,
  a result which is of independent interest. 
To make this precise, we recall the $q$-Pochhammer symbol 
$$(a; q)_k:=\displaystyle\prod_{n=0}^{k-1}(1 -aq^n
).$$

\begin{theorem}\label{PartitionAsymptotic}
If $A\geq 2$ is an integer, then the following are true.

\noindent
(1) We have that $p_k(A;n)$ is the coefficient of $T^{k}q^n$  in the infinite product
$$
\frac{(q^{A};q^{A})_{\infty}}{(q;q)_\infty(Tq^{A};q^{A})_\infty}.
$$

\smallskip
\noindent
(2) For every non-negative integer $n,$ we have $p_k(A;n)=p_{\leq k}(A;n-Ak).$
Moreover, we have
\[
\frac{(q^A;q^A)_{\infty}}{(q;q)_\infty(q^A;q^A)_k}=\sum_{n\geq0}p_{\leq k}(A;n)q^n.
\]

\smallskip
\noindent
(3) For fixed $k$, as $n\rightarrow +\infty,$ we have the asymptotic formulas
\begin{displaymath}
\begin{split}
&p_{\leq k}(A;n)\sim \frac{24^{\frac k2-\frac14}n^{\frac k2-\frac34}}{\sqrt2\left(1-\frac1A\right)^{\frac k2-\frac14}k!A^{k+\frac12}(2\pi)^k}e^{2\pi\sqrt{\frac1{6}\left(1-\frac1A\right)n}},\\
&p_{k}(A;n)\sim \frac{24^{\frac k2-\frac14}(n-Ak)^{\frac k2-\frac34}}{\sqrt2\left(1-\frac1A\right)^{\frac k2-\frac14}k!A^{k+\frac12}(2\pi)^k}e^{2\pi\sqrt{\frac1{6}\left(1-\frac1A\right)(n-Ak)}}.
\end{split}
\end{displaymath}
\end{theorem}
\begin{example} Here we illustrate the convergence of the asymptotic  for $p_1(3;n).$
Theorem~\ref{PartitionAsymptotic} (3) gives
\[
p_{1}(3;n)\sim \frac{1}{6\pi(n-3)^{\frac14}}e^{\frac{2\pi\sqrt{n-3}}{3}}.
\]
For convenience, we let $p^*_1(3;n)$ denote the right hand side of this asymptotic. The table below
illustrates the convergence of the asymptotic.
\bigskip
\begin{center}

\begin{tabular}{|c|cc|cc|cc|}
\hline \rule[-3mm]{0mm}{8mm}
$n$       && $p_1(3;n)$           && $p_1^*(3;n)$  & $p_1(3;n)/p_1^*(3;n)$ & \\   \hline 
$ 200$ &&  $93125823847$ && $\approx 82738081118$ & $\approx 1.126$ & \\
$ 400$ &&  $\approx 1.718\times 10^{16}$ && $\approx 1.579\times 10^{16}$ & $\approx 1.088$ & \\
$ 600$ &&  $\approx1.928\times 10^{20}$ && $\approx 1.799\times 10^{20}$ & $\approx 1.071$ & \\
$ 800$ &&  $\approx 5.058\times 10^{23}$ && $\approx 4.764\times 10^{23}$ & $\approx 1.062$ & \\
$ 1000$ &&  $\approx 5.232\times 10^{26}$ && $\approx 4.959\times 10^{26}$ & $\approx 1.055$ & \\
\hline
\end{tabular}
\captionof{table}{Asymptotics for $p_1(3;n)$}\label{table1}
\end{center}
\medskip

\end{example}

This paper is organized as follows.
In Section~\ref{ErdosLehnerSection} we prove Theorem~\ref{ELGeneralization}, the generalization of the classical limiting distribution (\ref{ELTheorem})  of Erd\H{o}s and Lehner.
In Section~\ref{ProofMain}, we recall the work of Buryak, Feigin, and Nakajima  \cite{BuryakFeigin, IMRN}, which gives the infinite product generating functions for the Poincar\'e polynomials
$P\left(X^{[n]}_{\alpha,\beta};T\right).$  These generating functions relate the Betti numbers to the partition functions $p_{\leq k}(\cdot).$
We use these facts, combined with Theorem~\ref{ELGeneralization}, to obtain Theorem~\ref{MainTheorem}.
 Finally, in Section~\ref{Partition} we obtain Theorem~\ref{PartitionAsymptotic}, the asymptotic formulas for  the $p_{\leq k}(A;n)$ and $p_k(A;n)$  partition functions. These asymptotics follow from an application of Ingham's Tauberian theorem. 

\section*{Acknowledgements}
The authors thank Kathrin Bringmann for helpful comments on an earlier version of this paper, and for pointing out the corrected version of the statement of Ingham's Tauberian Theorem. The authors also thank Ole Warnaar for comments on an earlier version of this paper.
 The second author thanks the support of the Thomas Jefferson Fund and the NSF (DMS-1601306 and DMS-2055118), and the Kavli Institute grant NSF PHY-1748958. The third author is grateful for the support of 
 a grant from the Simons Foundation (853830, LR) and a 2021-2023 Dean's Faculty Fellowship from Vanderbilt University.  Finally, the authors thank the referee for pointing out typographical errors in the original manuscript.
 
 \section{Generalization of a theorem of Erd\H{o}s and Lehner}\label{ErdosLehnerSection}
 
 Here we prove Theorem~\ref{ELGeneralization}. To prove the theorem we combine some elementary observations about integer partitions with delicate asymptotic analysis.
 
 \subsection{Elementary considerations}
 First we begin with an elementary convolution involving the partition functions
  $p_{\leq k}(A;\cdot )$, $p_{\leq k}(\cdot ),$ and $\preg(A;n)$, the number of $A$-regular partitions of size $n.$
  Recall that a partition is $A$-regular if all of its parts are not multiples of $A$.
 
 \begin{prop}\label{Convolution}
 If $A\geq 2$ is a positive integer, then for every positive integer $n$ we have
 $$
 p_{\leq k}(A;n)=\sum_{j=0}^{\lfloor \frac{n}{A}\rfloor} p_{\leq k}(j)\cdot \preg(A; n-Aj).
 $$
 \end{prop}
 \begin{proof}
Every partition of $n$ with at most $k$ parts that are multiples of $A$ can be represented as the direct product of an $A$-regular partition and a partition into at most $k$ parts that are all multiples of $A$.
If the sum of these multiples of $A$ is $Aj$, then the $A$-regular partition has size $n-Aj$. Moreover, by dividing by $A$, the multiples of $A$ are represented by a partition of $j$ into at most $k$ parts.
This proves the claimed convolution.
\end{proof}  

We also require an elegant inclusion-exclusion formula due to Erd\H{o}s and Lehner \cite{EL}
for $p_{\leq k}(n).$ 

\begin{prop}\label{ELIdentity}
If $k$ and $j$ are positive integers, then
$$
p_{\leq k}(j)=\sum_{m=0}^{\infty} (-1)^m S_k(m;j),
$$
where\footnote{The $S_k(m;j)/p(j)$ are denoted $S_m$ in \cite{EL}.}
\begin{equation}\label{Smj}
S_k(m;j):=\sum_{\substack{1\leq r_1<r_2<\dots <r_m\\
T_m\leq r_1+r_2+\dots+r_m\leq j-mk}} p\left(j-\sum_{i=1}^m (k+r_i)\right)
\end{equation}
and $T_m:=m(m+1)/2.$
\end{prop}
\begin{proof}
By definition, $p_{\leq k}(j)$ is the number of partitions of $j$ with at most $k$ parts. By considering conjugates of partitions, one can equivalently define $p_{\leq k}(j)$ as the number of partitions of $j$
with no parts $\geq k+1$.
Since the number of partitions of size $j$ that contain a part of size $k+r,$ where $r\geq 1$, equals
$p(j-(k+r)),$ we find that $S_k(1,j)$ is generally an overcount for the number of partitions of $j$ with at least one part $\geq k+1.$ Due to this overcounting, we find that
$$
p(j)-S_k(1;j)\leq p_{\leq k}(j) \leq p(j)-S_k(1;j)+S_k(2;j),
$$
which is obtained by taking into account those partitions which have at least two parts of distinct size $\geq k+1.$ The claim follows in this way by inclusion-exclusion.
\end{proof}

\subsection{Proof of Theorem~\ref{ELGeneralization}}
To prove Theorem~\ref{ELGeneralization}, we require Propositions~\ref{Convolution} and
~\ref{ELIdentity}, and the asymptotics for $\preg(A;n).$   Thanks to the identity
$$
\prod_{n=1}^{\infty}(1+q^n+q^{2n}+\dots + q^{(A-1)n}) =\prod_{n=1}^{\infty}\frac{(1-q^{An})}{(1-q^n)}=
\sum_{n=0}^{\infty} \preg(A;n)q^n,
$$
we find that $\preg(A;n)$ equals the number of partitions of $n$ where no part occurs more than $A-1$ times. Hagis \cite{Hagis} obtained asymptotics for
the number of partitions where no part is repeated more than $t$ times, and letting $t=A-1$ in Corollary 4.2 of \cite{Hagis} gives the following theorem.

\begin{theorem}\label{HagisTheorem}
If $A\geq 2,$ then we have
$$
\preg(A;n)=C_A (24n-1+A)^{-\frac{3}{4}}\exp\left(C\sqrt{\frac{A-1}{A}\left(n+\frac{A-1}{24}\right)}\right)
\left(1+O(n^{-\frac{1}{2}})\right),
$$
where $C:=\pi \sqrt{2/3}$ and $C_A:=\sqrt{12}A^{-\frac{3}{4}}(A-1)^{\frac{1}{4}},$ and the implied constant is independent of $A$.
\end{theorem}

\begin{proof}[Proof of Theorem~\ref{ELGeneralization}]
Thanks to Propositions~\ref{Convolution} and ~\ref{ELIdentity}, we have that
\begin{equation}\label{BigFormula}
\frac{p_{\leq k}(A;n)}{p(n)}
 =\sum_{j=0}^{\lfloor \frac{n}{A}\rfloor}
 \frac{\left(\sum_{m=0}^{\infty} (-1)^m S_k(m;j)\right) \preg(A;n-Aj)}{p(n)}.
\end{equation}
The proof follows directly from this expression by a sequence of observations involving the asymptotics for
$p(\cdot)$ and $\preg(A;\cdot),$ combined with the earlier work of Erd\H{o}s and Lehner on the sums $S_k(m;j).$
Thanks to the special choice of $k_n=k_n(x)$, this expression yields the Taylor expansion of the claimed cumulative Gumbel distribution in $x,$ as $n\rightarrow +\infty.$ In other words, these asymptotics conspire so that the dependence on $n$
vanishes in the limit.

 For convenience, we let $S^*_k(m;j):=S_k(m;j)/p(j).$
In terms of $S^{*}_k(m,j),$ 
 \eqref{BigFormula} becomes
\begin{equation}\label{BigFormula2}
\frac{p_{\leq k}(A;n)}{p(n)}
 =\sum_{j=0}^{\lfloor \frac{n}{A}\rfloor}
 \frac{\left(\sum_{m=0}^{\infty} (-1)^m  S^*_k(m;j)\right) p(j)\preg(A;n-Aj)}{p(n)}.
\end{equation}
To make use of this formula, we begin by employing  the method of Erd\H{o}s and Lehner {\it mutatis mutandis}, which we briefly
recapitulate here.
 For $k\rightarrow +\infty,$ with $j$ and $m$ fixed, Erd\H{o}s and Lehner proved
(see (2.5) of \cite{EL}) that
\begin{equation}\label{Sasymptotics}
S^*_k(m;j) =\frac{1}{m!}\left(\frac{2}{C}\sqrt{j} \exp\left(-\frac{C}{ 2\sqrt{j}}k  \right)\right)^m  +o_{j,m}(1).
\end{equation}
For every positive integer $m,$ this effectively gives
$$
S^*_k(m;j)=\frac{1}{m!}\cdot S^*_k(1;j)^m + o_{j,m}(1)\sim \frac{1}{m!}\cdot S_k^*(1;j)^m,
$$
which Erd\H{os} and Lehner show produces, as functions in $x$, the asymptotic
 \begin{equation}\label{SumSk}
\sum_{m=0}^{\infty} (-1)^m  S^*_{k_n}(m;j) =\exp(-S^*_{k_n}(1;j))\left(1+ o_{n}(1)\right).
\end{equation}
We recall the choice of $k=k_{A,n}=k_{A,n}(x)=\frac{1}{AC} \sqrt{n} \log(n) +x \sqrt{n}.$
This is the exponential which arises in the exponential of the claimed cumulative distribution.

 To make use of (\ref{SumSk}), 
 it is convenient to recenter the sum on $j$  in  (\ref{BigFormula2}) by setting $j=\frac{n}{A^2} +y.$
 As (\ref{SumSk}) only involves $S^*_{k_n}(1;j),$ it suffices to note that when $m=1,$
 \eqref{Sasymptotics} becomes
\begin{align}\label{AsymptoticsS1}
S^*_{{\color{black}k_{A,n}}}(1;j) &=\frac{2}{AC}\sqrt{n+A^2y}\cdot  \exp\left(-\frac{\log(n)}{ 2\sqrt{1+yA^2/n}}  -\frac{xAC}{ 2\sqrt{1+yA^2/n}}  \right)  +o_{n}(1).
\end{align}
As the proof relies on (\ref{BigFormula2}), we must also estimate the quotients
$$
\frac{p(j)\preg(A;n-Aj)}{p(n)}.
$$
Thanks to the Hardy-Ramanujan asymptotic for $p(n)$ and Theorem~\ref{HagisTheorem}, we have
\begin{align}\label{PratioAsymptotics}
&\nonumber\frac{p(j)\preg(A;n-Aj)}{p(n)}\\
&  =\frac{C_A}{(24n-24Aj-1+A)^{\frac{3}{4}}} \frac{n}{j}\exp{\left(C\left(\sqrt{j}-\sqrt{n}+\sqrt{\frac{A-1}{A}\left(n-Aj+\frac{A-1}{24}\right)}\right)\right)}\cdot\left(1+O_{j}(n^{-\frac{1}{2}})\right)\notag\\
& =\frac{C_A}{(24n-24n/A-24Ay-1+A)^{\frac{3}{4}}} \frac{A^2n}{n+A^2y}\notag\\
&\ \ \ \ \ \ \ \ \ \ \times \exp{\left(C\left(\sqrt{n/A^2+y}-\sqrt{n}+\sqrt{\frac{A-1}{A}\left(n-n/A-Ay+\frac{A-1}{24}\right)}\right)\right)}
\cdot\left(1+O_{y }(n^{-\frac{1}{2}})\right).
\end{align}
The last manipulation uses the change of variable for $j$. 

We will make use of (\ref{SumSk}), (\ref{AsymptoticsS1}) and (\ref{PratioAsymptotics}) to complete the proof.
To this end, we let
$j=\lfloor n/A^2\rfloor+y$ essentially as above, but now modified\footnote{We can ignore the difference between $\lfloor n/A^2\rfloor$ with $n/A^2$ as it makes no difference for our limit calculations.} so that the $y$ are integers.  We then rewrite \eqref{BigFormula2} as 
\[\frac{p_{\leq {\color{black}k_{A,n}}}(A;n)}{p(n)} =\Sigma_{1}+\Sigma_{2}+\Sigma_{3},
\]
where $\Sigma_{1}$ is the sum over $-n/A^2\leq y< -n^{3/4}\log(n)$, $\Sigma_{2}$ is the sum over $-n^{3/4}\log(n)\leq y\leq  n^{3/4}\log(n),$
 and $\Sigma_{3}$ is the sum over $n^{3/4}\log(n)\leq y\leq n(1/A-1/A^2).$ 
 We shall show that the main contribution will come from $\Sigma_2,$ and that  
$\Sigma_1$ and $\Sigma_3$ vanish as $n\rightarrow +\infty.$

To establish the vanishing of $\Sigma_1+\Sigma_3,$ we consider the case that $|y|> n^{3/4}\log(n).$ For such $y$ we have
\[
\sqrt{n/A^2+y}-\sqrt{n}+\sqrt{\frac{A-1}{A}\left(n-n/A-Ay+\frac{A-1}{24}\right)}= O_y(\sqrt{n}),
\]
where the implied constant is negative. Moreover, \eqref{AsymptoticsS1} implies that $S^*_{{\color{black}k_{A,n}}}(1;n/A^2+y)=O(\sqrt{n}),$ where the implied constant is positive. Thus, for $y$ in these ranges,  both $\frac{p(j)}{p(n)} \preg(A;n-Aj)$ and $\sum_{m=0}^{\infty} (-1)^m  S^*_{{\color{black}k_{A,n}}}(m;j)$ decay sub-exponentially, and so 
\[\lim_{n\to \infty} \Sigma_1+\Sigma_3=0.\]

We now consider $\Sigma_2$, where $|y|\leq n^{3/4}\log(n).$ In this range, \eqref{AsymptoticsS1} becomes 
\begin{align}\label{AsymptoticsS1Again}
S^*_{{\color{black}k_{A,n}}}(1;j) &=\frac{2}{AC}\sqrt{n+A^2y}\cdot \exp\left(-\frac{\log(n)}{ 2\sqrt{1+yA^2/n}}  -\frac{xAC}{ 2\sqrt{1+yA^2/n}}  \right)  +o_{n}(1)\\
 &=\frac{2}{AC} \exp\left( -\frac 12 xAC  \right) +o_{n}(1).
\end{align}
Using (\ref{SumSk}), we obtain
\begin{equation}\label{SumSk2}
\sum_{m=0}^\infty (-1)^m S^*_{{\color{black}k_{A,n}}}(m;j)=\exp\left (-\frac{2}{AC} \exp\left( -\frac 12 xAC  \right)\right)\left(1+o_{n}(1)\right).
\end{equation}
We now estimate (\ref{PratioAsymptotics}) for these $|y|\leq  n^{3/4}\log(n).$
Since we have
\begin{align*}
\sqrt{n/A^2+y}-\sqrt{n}+\sqrt{\frac{A-1}{A}\left(n-n/A-Ay+\frac{A-1}{24}\right)} &=
-\frac{A^4}{8(A-1)}y^2n^{-3/2}+O_A(y^3n^{-5/2}),
\end{align*}
the hypothesis on $y$ allows us to turn (\ref{PratioAsymptotics}) into 
\[
\frac{p(j)\preg(A;n-Aj)}{p(n)} = 
\frac{A^2 C_A}{(24n\frac{A-1}{A})^{3/4}}\times \exp{\left(-C\frac{A^4}{8(A-1)}\frac{y^2}{n^{3/2}}
\right)}
\cdot\left(1+O_{A }(n^{-\frac{1}{4}+\varepsilon})\right).
\]
Combined with \eqref{SumSk2}, and using $C_A=\sqrt{12}A^{-\frac{3}{4}}(A-1)^{\frac{1}{4}},$ we obtain
\begin{displaymath}
\begin{split}
&\lim_{n\to\infty}\Sigma_2\\
&\ \ =\lim_{n\to \infty }\sum_{|y|<n^{3/4}\log(n)} \frac{A^2}{96^{1/4}\sqrt{A-1}}\cdot \frac{1}{n^{3/4}}\cdot
 \exp\left (-\frac{CA^4}{8(A-1)}\frac{y^2}{n^{3/2}}-\frac{2}{AC} \exp\left( -\frac 12 xAC 
\right)\right)
\cdot\left(1+o_{A }(1)\right).
\end{split}
\end{displaymath}
Approximating the right hand side as a Riemann sum, we obtain
\begin{equation}
\lim_{n\rightarrow +\infty} \Sigma_2=\lim_{n\rightarrow +\infty} \frac{A^2}{96^{1/4}\sqrt{A-1}}\int_{-\log(n)}^{\log(n)}
 \exp\left (-\frac{CA^4}{8(A-1)}t^2-\frac{2}{AC} \exp\left( -\frac 12 xAC 
\right)\right) dt,
\end{equation}
where $n$ only appears in the limits of integration.
To obtain this, we have used the substitutions $t=yn^{-3/4}$ and $dt=n^{-3/4}dy,$ and employ the fact
that the widths of the subintervals defining the Riemann sums tend to 0. 
Expanding as an integral over $\R$, this expression simplifies to
\[
\exp\left (
-\frac{2}{AC} \exp\left( -\frac 12 xAC 
\right)\right).
\]
Therefore, as a function in $x,$ we have
$$
\lim_{n\rightarrow +\infty}\frac{p_{\leq k_{A,n}}(A;n)}{p(n)}= \exp\left (
-\frac{2}{AC} \exp\left( -\frac 12 xAC 
\right)\right).
$$
This completes the proof of the theorem.

\end{proof}

\section{Application to the Hilbert schemes $X_{\alpha,\beta}^{[n]}$}\label{ProofMain}
Here we recall the relevant generating functions for the Poincar\'e polynomials of the Hilbert schemes that pertain to Theorem~\ref{MainTheorem}.
For the various Hilbert schemes on $n$ points, G\"ottsche, Buryak, Feigin, and
Nakajima \cite{BuryakFeigin, IMRN, Gottsche, GottscheICM} proved infinite product generating functions
for these Poincar\'e polynomials. For Theorem~\ref{MainTheorem}, we require the following theorem.

\begin{theorem}{\text {\rm (Buryak and Feigin)}}\label{QuasiHilbertGenFcn}
If $\alpha, \beta\in \N$ are relatively prime, then we have that
$$
G_{\alpha,\beta}(T;q):=\sum_{n=0}^{\infty} P\left(X^{[n]}_{\alpha,\beta};T\right)q^n=\frac{(q^{\alpha+\beta};q^{\alpha+\beta})_{\infty}}{(q;q)_\infty(T^2q^{\alpha+\beta};q^{\alpha+\beta})_\infty.
}
$$
\end{theorem}

\begin{remark}
The Poincar\'e polynomials in these cases only have even degree terms (i.e. odd index
Betti numbers are zero). Moreover, letting $T=1$ in these generating functions give Euler's generating function for $p(n).$ Therefore, we directly see that
$$
p(n)=P\left(X^{[n]}_{\alpha,\beta};1\right).
$$
Of course, the proof of Theorem~\ref{QuasiHilbertGenFcn} begins with partitions of size $n$.
\end{remark}

\begin{corollary}\label{DiscreteMeasure}
Assuming the  notation and hypotheses above, if $d\mu^{[n]}_{\alpha,\beta}$ is the discrete measure for
$X^{[n]}_{\alpha,\beta}$, then 
\[
\Phi_n(\alpha, \beta; x)=\frac{1}{p(n)}\cdot \int_{-\infty}^x d\mu^{[n]}_{\alpha,\beta}=
\frac{p_{\leq \frac{x}{2}}(\alpha+\beta;n)}{p(n)}.
\]
\end{corollary}
\begin{proof}
By Theorem~\ref{QuasiHilbertGenFcn}, the Poincar\'e polynomial $P\left(X^{[n]}_{\alpha,\beta};T\right)$ is the coefficient of $q^n$ of 
\[
\frac{(q^{\alpha+\beta};q^{\alpha+\beta})_{\infty}}{(q;q)_\infty(T^2q^{\alpha+\beta};q^{\alpha+\beta})_\infty}.
\]
Part (1) of Theorem~\ref{PartitionAsymptotic} applied to $A=\alpha+\beta$ gives that the coefficient of $T^{2k}$ in this expression is $p_k(\alpha+\beta;n)$ (the odd powers of $T$ do not appear in this product as it is a function of $T^2$). Therefore, (\ref{PoincareDefn}) becomes
$$
P\left(X^{[n]}_{\alpha,\beta};T\right)=\sum_{j=0}^{\lfloor \frac{n}{\alpha+\beta}\rfloor} p_j(\alpha+\beta;n) T^{2j}=
\sum_{j=0}^{2\lfloor \frac{n}{\alpha+\beta}\rfloor} \dim \left(H_j\left(X^{[n]}_{\alpha,\beta},\Q\right)\right)T^j.
$$
Thus, the sum of coefficients up to $x$, divided by $p(n)$, is
\[
\frac{1}{p(n)}\cdot \sum_{j\leq x}b_j(\alpha,\beta; n)=\frac{1}{p(n)}\cdot \sum_{j\leq x/2}p_j(\alpha+\beta;n)=\frac{p_{\leq \frac{x}{2}}(\alpha+\beta; n)}{p(n)}.
\]
This completes the proof.
\end{proof}

\begin{proof}[Proof of Theorem~\ref{MainTheorem}]
To prove Theorem~\ref{MainTheorem}, we remind the reader that
Theorem~\ref{ELGeneralization} gives the cumulative asymptotic distribution function for
 $p_{\leq k}(A;n)$ when $A\geq 2.$
Corollary~\ref{DiscreteMeasure}, with $A=\alpha+\beta,$ identifies this partition distribution with the Betti distribution for the $n$ point Hilbert schemes cut out by the $\alpha, \beta$ torus action. The theorem follows by combining these two results.
\end{proof}

 \section{Asymptotic formulae for the $p_k(A;n)$ partition functions}\label{Partition}

Here we prove Theorem~\ref{PartitionAsymptotic}. To this end, we make use of
 Ingham's Tauberian theorem \cite{Ingham}. We note that this theorem is misstated in a number of places in the literature. Condition (3) in the statement below is often omitted. The reader is referred to the discussion in \cite{BJS}. Here we use a special case\footnote{In the notation of \cite{BJS}, we let $d=\beta$, $N=\gamma$, and we let $\alpha=0$ in the case of weak monotonicity of Theorem 1.1. } of Theorem 1.1 of \cite{BJS}.
\begin{theorem}[Ingham]\label{TaubThm}
Let $f(q)=\sum_{n\geq0}a(n)q^n$
be a holomorphic function in the unit disk \\ \noindent $|q|<1$  satisfying the following conditions:

\noindent
(1) The sequence $\left\{a(n)\right\}_{n\geq0}$ is positive and weakly monotonically increasing.

\smallskip
\noindent
(2) There exist $c\in\C,$ $d\in\R,$ and $N>0,$ such that as $t\rightarrow 0^+$ we have
\[
f(e^{-t}) \sim \lambda\cdot t^{d}\cdot e^{\frac{N}{t}}.
\]
(3) For any $\Delta>0$, in the cone $|y|\leq\Delta x$ with $x>0$ and $z=x+iy$, we have, as $z\rightarrow0$
\[
f(e^{-z})\ll |z|^d\cdot e^{\frac{N}{|z|}}.
\]
\noindent
Then as $n\rightarrow +\infty$ we have
\[
a(n)\sim\frac{\lambda\cdot N^{\frac d2+\frac14}}{2\sqrt\pi \cdot n^{\frac d2+\frac34}}e^{2\sqrt{Nn}}.
\]
\end{theorem}

\begin{proof}[Proof of Theorem~\ref{PartitionAsymptotic}] We prove the claims one-by-one.

\smallskip
\noindent
(1) We begin by recalling the
$q$-Pochhammer symbol 
$$(a; q)_k:=\displaystyle\prod_{n=0}^{k-1}(1 -aq^n
).$$
Clearly, we have
\[
 \frac{(q^{A};q^{A})_{\infty}}{(q;q)_\infty}=\prod_{j=1}^{A-1}\frac1{(q^j;q^A)_\infty}, 
\]
which in turn gives
\[
\frac{(q^{A};q^{A})_{\infty}}{(q;q)_\infty(Tq^{A};q^{A})_\infty}=\prod_{n\not\equiv0\pmod A}\frac1{1-q^n}\quad \times\prod_{n\equiv0\pmod A}\frac{1}{1-Tq^n}.
\]
Expanding each term as a geometric series, we find that the coefficient of $T^{k}$ collects those partitions which have $k$ parts which are $0\pmod A$. 

\smallskip
\noindent
(2) We make use of the $q$-binomial theorem, which asserts that
$$
\sum_{n\geq0}\frac{(a;q)_n}{(q;q)_n}z^n=\frac{(az;q)_{\infty}}{(z;q)_{\infty}}.
$$
Hence, if we let 
 $[T^{k}]$  denote the coefficient of $T^{k}$, this theorem allows us to conclude that
\[
\frac{(q^{A};q^{A})_{\infty}}{(q;q)_\infty}[T^{k}]\left(\frac{1}{(Tq^A;q^A)_{\infty}}\right)=\frac{(q^{A};q^{A})_{\infty}}{(q;q)_\infty}[T^{k}]\left(\sum_{n\geq0}\frac{(Tq^A)^n}{(q^A;q^A)_n}\right)=\frac{q^{Ak}(q^{A};q^{A})_{\infty}}{(q;q)_\infty(q^A;q^A)_k}.
\]
Arguing as in the proof of (1), we find the claimed generating function identity
\begin{equation}\label{qseries}
\frac{(q^{A};q^{A})_{\infty}}{(q;q)_\infty(q^A;q^A)_k}=\sum_{n\geq0}p_{\leq k}(A;n)q^n.
\end{equation}
These two $q$-series identities, combined with (1), imply that $p_k(A;n)=p_{\leq k}(A;n-Ak).$

\smallskip
\noindent
(3)  To establish the desired asymptotics, we apply Theorem~\ref{TaubThm} to (\ref{qseries}),
which is facilitated by the modularity of
Dedekind's eta-function
$$
\eta(\tau):=q^{\frac{1}{24}}(q;q)_{\infty}.
$$
 This function is well-known to satisfy
 $$
\eta\left(-\frac{1}{\tau}\right)=\sqrt{\frac{\tau}{i}}\eta(\tau).
$$
As a consequence of this transformation and the $q$-expansion $\eta(\tau)=q^{\frac1{24}}+O(q^{\frac{25}{24}})$ near $\tau=i\infty$  (for example, see p. 53 of \cite{Zagier}), for $q=e^{-t}$, $t\rightarrow0^+,$ we find that
 \begin{equation}\label{eta}
\log\left(\frac1{(q;q)_{\infty}}\right)=\frac{\pi^2}{6t}-\frac12\log\left(\frac{2\pi}{t}\right)+O(t).
 \end{equation}
 Thus, letting $t\mapsto At$ and taking a difference yields 
 \begin{equation}\label{nextfactor}
 \log\left(\frac{(q^A;q^A)_{\infty}}{(q;q)_{\infty}}\right)=\frac{\pi^2}{6t}\left(1-\frac1A\right)-\frac{\log(A)}{2}+O(t).
 \end{equation}
This calculation gives the behavior in the radial limit as $t\rightarrow0^+$ of the infinite Pochhammer symbols in \eqref{qseries}. 

To satisfy condition (3) of Theorem~\ref{TaubThm}, we also need to estimate the quotient on the left hand side of \eqref{nextfactor} for the regions $|y|\leq \Delta x$. This is given directly in Section~3.1 of \cite{BJS}. Namely, they show that in these regions, one has
\[
\frac{1}{(e^{-z};e^{-z})_{\infty}}=\sqrt{\frac{z}{2\pi}}\cdot\ e^{\frac{\pi^2}{6z}}\left(1+O_{\Delta}\left(\left|e^{-\frac{4\pi^2}{z}}\right|\right)\right)
\]
and
\[
e^{-\frac{1}{z}}\leq e^{-\frac{1}{(1+\Delta^2)|z|}}.
\]
Thus, we have
\begin{align}\label{qid}
\frac{1}{(e^{-z};e^{-z})_{\infty}}=\sqrt{\frac{z}{2\pi}}\cdot\ e^{\frac{\pi^2}{6z}}\left(1+O_{\Delta}\left(e^{-\frac{4\pi^2}{(1+\Delta^2)|z|}}\right)\right).
\end{align}
Changing variables to let $z\mapsto Az$, we then find
\begin{equation}\label{InfiniteCone}
\frac{(e^{-Az};e^{-Az})_\infty}{(e^{-z};e^{-z})_{\infty}}=\sqrt A\cdot e^{-\frac{\pi^2}{6z}\left(1-\frac1A\right)}\cdot\frac{\left(1+O_{\Delta}\left(e^{-\frac{4\pi^2}{A(1+\Delta^2)|z|}}\right)\right)}{\left(1+O_{\Delta}\left(e^{-\frac{4\pi^2}{(1+\Delta^2)|z|}}\right)\right)}=\sqrt A\cdot e^{-\frac{\pi^2}{6z}\left(1-\frac1A\right)}\left(1+O_\Delta\left(e^{-\frac{4\pi^2}{A(1+\Delta^2)|z|}}\right)\right).
\end{equation}

 Now we turn to estimating the remaining factor in \eqref{qseries}, namely, $1/(q^A;q^A)_k$. 
On the line $t\rightarrow0^+$, an important result of Zhang 
(see Theorem 2 of  \cite{Zhang})  gives that for $0<t\rightarrow0$ and $w\in\C$, 
 $$
(e^{-wt};e^{-t})_{\infty}\sim \frac{\sqrt{2\pi}}{\Gamma(w)}e^{-\frac{\pi^2}{6t}-\left(w-\frac12\right)\log(t)}.
 $$
Letting $w=k+1$ and combining with (\ref{eta}), we conclude that
 \[
 \frac{1}{(q;q)_{k}}=\frac{(q^{k+1};q)_{\infty}}{(q;q)_{\infty}}
 \sim\frac{\sqrt 2\pi}{k!}e^{-\frac{\pi^2}{6\varepsilon}-(k+1/2)\log(t)+\frac{\pi^2}{6t}-\frac12\log(2\pi/t)}=\frac{t^{-k}}{k!}.
 \]
 Letting $t\mapsto At$, we have 
 \begin{equation}\label{FiniteCond2}
 \frac{1}{(q^A;q^A)_{k}}\sim\frac{1}{k!A^k}t^{-k}.
 \end{equation}
 
 Turning to estimate $1/(q^A;q^A)_k$ in the regions $|y|\leq\Delta x$,  we use the same argument in the proof of Theorem 2 of \cite{Zhang}. One merely modifies the proof by replacing $x$ with $|z|$ in Zhang's setting to obtain
 \[
 {(e^{-A(k+1)z};e^{-Az})_{\infty}}\ll\frac{\sqrt{2\pi}}{k!}e^{-\frac{\pi^2}{6|z|}-\left(k+1-\frac12\right)\log|z|},
 \]
 as $z\rightarrow 0.$
 Moreover, by combining with (\ref{qid}), we have
 \begin{equation}\label{FiniteCond3}
 \frac{1}{(e^{-Az};e^{-Az})_{k}}=\frac{(e^{-A(k+1)z};e^{-Az})_{\infty}}{(e^{-z};e^{-z})_{\infty}}\ll\frac{|z|^{-k}}{k!}.
 \end{equation}
  Then multiplying \eqref{InfiniteCone} and \eqref{FiniteCond3}, we find that 
 \begin{equation}\label{InghamShape}
 \frac{(e^{-Az};e^{-Az})_{\infty}}{(e^{-z};e^{-z})_{\infty}(e^{-Az};e^{-Az})_k}\ll \frac{\sqrt{A}}{k!}|z|^{-k}e^{\frac{\pi^2}{6|z|}\left(1-\frac{1}{A}\right)} ,
 \end{equation}
 which shows that condition (3) of Theorem~\ref{TaubThm} is satisfied. 
 
 Multiplying \eqref{nextfactor} with \eqref{FiniteCond2}, where $q:=e^{-t}$, we obtain
\[
\frac{(q^A;q^A)_{\infty}}{(q;q)_{\infty}(q^A;q^A)_{k}}\sim
\frac{1}{k!A^{k+\frac12}}t^{-k}e^{\frac{\pi^2}{6t}\left(1-\frac 1A\right)}.
\]

Moreover, the coefficients $\frac{(q^A;q^A)_{\infty}}{(q;q)_{\infty}(q^A;q^A)_{k}}$ are clearly positive as they count partitions. They are weakly increasing as there is an easy injection from the set of partitions of $n$ with at most $k$ parts which are multiples of $A$ into the set of partitions of $n+1$ which have at most $k$ parts which are multiples of $A$; simply add $1$ to the partition, which doesn't affect the number of multiples of $A$ among the parts. 

We are thus in the situation of Theorem~\ref{TaubThm}, where we interprete (\ref{InghamShape}) with
\[
\lambda=\frac{1}{k!A^{k+\frac12}},\quad
d=-k,
\quad
N=\frac{\pi^2}{6}\left(1-\frac1A\right).
\]
Plugging these into the Theorem~\ref{TaubThm} gives the desired asymptotic for $p_{\leq k}(A;n).$ The asymptotics for $p_k(A;n)$ follows from the identity $p_k(A;n)=p_{\leq k}(A;n-Ak)$ obtained in (2).

\end{proof}

\end{document}